\documentclass[12pt]{article}

\usepackage{amsmath,amsthm, amssymb, latexsym}
\usepackage{amsfonts}
\usepackage{graphicx}
\usepackage{multicol}
\usepackage{caption}
\usepackage{color}
\usepackage{amsmath}
\usepackage{multicol}
\usepackage{float}
\usepackage{bbold}
\usepackage{color}
\usepackage{subfigure}
\usepackage{url}
\usepackage{tabularx}
\usepackage{xcolor}
\usepackage{grffile}
\usepackage{natbib}
\usepackage[english]{babel}
\usepackage{enumerate}

 \normalbaselines
\newtheorem{Theorem}{Theorem}[section]
\newtheorem{Definition}[Theorem]{Definition}
\newtheorem{Proposition}[Theorem]{Proposition}
\newtheorem{Lemma}[Theorem]{Lemma}
\newtheorem{Corollary}[Theorem]{Corollary}

\newtheorem{Remark}[Theorem]{Remark}
\newtheorem{Example}[Theorem]{Example}

\newcommand{\N}{\mathbb{N}}
\newcommand{\Z}{\mathbb{Z}}
\newcommand{\R}{\mathbb{R}}

\newcommand{\E}{\mathbb{E}}

\begin{document}

\title{Inheritance of strong mixing and weak dependence under renewal sampling}
\author{\small{Dirk-Philip Brandes, Imma Valentina Curato and Robert Stelzer\footnote{Ulm University, Institute of Mathematical Finance, Helmholtzstra\ss e 18, 89069 Ulm, Germany. E-mails: dirk-philip.brandes@alumni.uni-ulm.de, imma.curato@uni-ulm.de, robert.stelzer@uni-ulm.de. }}}

\maketitle

\textwidth=160mm \textheight=225mm \parindent=8mm \frenchspacing
\vspace{3mm}

\begin{abstract}
Let $X$ be a continuous-time strongly mixing or weakly dependent process and $T$ a renewal process independent of $X$ with inter-arrival times $\tau$. We show general conditions under which the sampled process $(X_{T_i},T_i-T_{i-1})^{\top}$ is strongly mixing or weakly dependent. Moreover, we explicitly compute the strong mixing or weak dependence coefficients of the renewal sampled process and show that exponential or power decay of the coefficients of $X$ is preserved (at least asymptotically). Our results imply that essentially all central limit theorems available in the literature for strongly mixing or weakly dependent processes can be applied when renewal sampled observations of the process $X$ are at disposal. 
\end{abstract}


\section*{Introduction}
Time series are ubiquitous in many applications, and it is often the case that the time interval separating two successive observations of the latter is itself random.  We approach the study of such time series by using a continuous-time stationary process $X=(X_t)_{t \in \R}$ and a renewal process $T=(T_i)_{i \in \Z}$ which models the sampling scheme applied to $X$. We assume that $X$ is strongly mixing or weakly dependent as defined, respectively, in \cite{R56} and \cite{DD08} and that $T$ is a process independent of $X$ with inter-arrival time sequence $\tau=(\tau_i)_{i\in \Z\setminus \{0\}}$. In this general model set-up, we show under which assumptions the renewal sampled process $Y=(Y_i)_{i \in \Z}$ defined as $Y_i=(X_{T_i},T_i-T_{i-1})^{\top}$ inherits strong mixing and weak dependence.

In the literature, the statistical inference methodologies based on renewal sampled data seldom employ a strongly mixing or weakly dependent process $Y$. To the best of our knowledge, the only existing example of this approach can be found in \cite{ASM04} where it is shown that $Y$ is $\rho$-strongly mixing, and this property is used to study the consistency of maximum likelihood estimators for continuous-time diffusion processes.  On the contrary, there exist several statistical estimators whose asymptotic properties heavily rely on ad hoc tailor-made arguments for specific models $X$. Examples of the kind appears in spectral density estimation theory. \cite{LM92}, \cite{M78a}, \cite{M78b}, and \cite{M83} study non-parametric and parametric estimators of the spectral density of $X$ employing an aliasing-free sampling scheme defined through a renewal process, see \cite{LM92} for a general definition of this set-up. Such schemes like the Poisson one allow overcoming the aliasing problem, typically observed when working with a not band limited process. Moreover, the spectral density estimators determined using renewal sampled data are consistent and asymptotically normally distributed once assumed that $X$ has finite moments of all orders. Renewal sampled data are also used to define a spectral density estimator for Gaussian processes in \cite{BB10}, kernel density estimators for strongly mixing processes in \cite{M88}, non-parametric estimators of volatility and drift for scalar diffusion in \cite{CT16}, and parametric estimators of the covariance function of $X$ as in \cite{MW79} and \cite{BC18}. In the latter, an estimator of the covariance function of a Gauss-Markov process and a continuous-time L\'evy driven moving average are respectively analyzed. In \cite{BC18}, in particular, the asymptotic properties of the estimator are obtained by an opportune truncation of a L\'evy driven moving average process $X$ that is proven to be strongly mixing. 

Determining conditions under which the process $Y$ inherits the asymptotic dependence of $X$ significantly widens the applicability of renewal sampled data. Just as indicative examples, our analysis should enable the use of renewal sampled data to study spectral density estimators as in \cite{R84}, Whittle estimators as in \cite{BD08}, and generalized method of moments estimators as in \cite{CS19} and \cite{SS19}. Moreover, the knowledge of the asymptotic dependence of $Y$ allows to apply well-established asymptotic results for $\alpha$-mixing processes like the ones in Chapter 10 of \cite{B07}, \cite{DR00}, and \cite{K17}. The latter are, respectively, functional and triangular array central limit theorems. The same argument applies to central limit theorems for weakly dependent processes like the ones presented in \cite{BS05}, \cite{DD03}, and \cite{DW07}. In brief, understanding the dependence structure of the process $Y$ allows to obtain joint asymptotic results for $(X_{T_i},T_i-T_{i-1})_{i \in \Z}$ which enable inference on the process $X$ also when the distribution of the sequence $\tau$ is not known, i.e., when the sampling scheme is not designed by an experimenter but just observed from the data. An example of the latter application appears in Theorem 5.2 of \cite{BC18}.

We study in this paper the inheritance of $\eta$, $\lambda$, $\kappa$, $\zeta$, $\theta$-weak dependence and $\alpha$-mixing which are extensively analyzed in the monographs of \cite{DD08}, \cite{B07} and \cite{D94}, respectively. For any positive integer $u,v$, indexes $i_1 \leq \ldots \leq i_u < i_u+r \leq j_1 \ldots \leq j_v$ with $r > 0$, and functions $F$ and $G$ belonging to specific functional spaces-- i.e., $F$ and $G$ are bounded Lipschitz or bounded measurable functions -- weakly dependent and $\alpha$-mixing processes both satisfy covariance inequalities of the following type,
\[
|Cov(F(X_{i_1},\ldots,X_{i_u}),G(X_{j_1},\ldots,X_{j_v}))|\leq c \,\Psi(\|F\|_{\infty},\|G\|_{\infty},Lip(F),Lip(G),u,v) \, \epsilon(r),
\]
where the sequence of coefficients $\epsilon=(\epsilon(r))_{r \in \R^+}$ converges to zero at infinity, $c$ is a constant independent of $r$, and the function $\Psi(\cdot)$ has different shapes depending on the functional spaces where $F$ and $G$ are defined. Hence, throughout, we use a unified formulation of weak dependence and $\alpha$--mixing. We call a process $\Psi$-weakly dependent if it satisfies such a covariance inequality, and we call $\epsilon$ the sequence of the $\Psi$-coefficients. Note that such a sequence of coefficients corresponds to weak dependence or $\alpha$-mixing coefficients -- as defined in Section 2.2 of \cite{DD08} and Definition 3.5 of \cite{B07} -- depending on the function $\Psi(\cdot)$.

Some formulae for $\alpha$, $\beta$, $\phi$ and $\rho$-mixing coefficients for the process $(X_{T_i})_{i \in \Z}$ have been previously obtained by \cite{CR07}, but their line of proof does not automatically extend to weak dependence, see Remark \ref{charlot} for more details. Moreover, the authors do not show when the convergence to zero of the coefficients is attained and that $(X_{T_i})_{i \in \Z}$ actually inherits strong mixing of $X$.  In Theorem \ref{main}, we give a general proof for the computation of the $\Psi$-coefficients related to the renewal process $Y$, which applies to weakly dependent and $\alpha$-mixing processes alike. Moreover, we present several sampling schemes for which the convergence to zero of the $\Psi$-coefficients is realized, and then the renewal process $Y$ inherits the dependence structure of $X$. In particular, under the additional condition that $X$ admits exponential or power decaying coefficients, we show that the $\Psi$-coefficients related to $Y$ preserve the exponential or power decay (at least asymptotically). 

The paper is organized as follows. We present in Section \ref{sec1} the definition of $\Psi$-weakly dependent processes, which encompasses weakly dependent and $\alpha$-mixing processes. In Section \ref{sec2}, we explicitly compute the $\Psi$-coefficients of the process $Y$. Moreover, we present data sets for which the independence between a process $T$, modeling the random sampling scheme, and $X$ is realistic.
Finally, in Section \ref{sec3}, we show that if the underlying process $X$ admits exponential or power decaying $\Psi$-coefficients, then the process $Y$ is $\Psi$-weakly dependent and has coefficients with (at least asymptotically) the same decay. This section includes several examples of renewal sampling. In particular, Poisson sampling times are discussed. Section \ref{sec4} concludes.

\section{Weak dependence and strong mixing conditions}
\label{sec1}
We assume that all random variables and processes are defined on a given probability space $(\Omega, \mathcal{A},\mathbb{P})$.

We refer by $\N^*$ to the set of positive integers, by $\N$ to the set of the non-negative integers, by $\Z$ to the set of all integers, and by $\R_+$ to the set of the non-negative real numbers. 
We denote the Euclidean norm by $\| \cdot \|$.  However, due to the equivalence of all norms, none of our results depends on the chosen norm.

Although the theory developed below is most relevant for sampling processes defined in continuous time, we work with a general index set $\mathcal{I}$ as this makes no difference and also covers other cases, like a sampling of discrete-time processes or random fields.
\begin{Definition}
	\label{index}
	The index set $\mathcal{I}$ is denoting either $\Z$, $\R$, $\Z^m$ or $\R^m$. Given $H$ and $J \subseteq \mathcal{I}$, we define $d(H,J)=min\{ \|i-j\|, i\in H,j\in J\}$.
\end{Definition}

Even if our theory extends to random fields, we always refer to $X$ as a process to lighten the reading. 

Moreover, we consider
\begin{equation}
	\label{set}
	\mathcal{F}=\bigcup_{u \in \N} \mathcal{F}_u \,\,\, \quad \textrm{and} \quad \,\,\, \mathcal{G}=\bigcup_{v \in \N} \mathcal{G}_v
\end{equation}
where $\mathcal{F}_u$ and $\mathcal{G}_v$ are respectively two classes of measurable functions from $(\R^d)^u$ to $\R$ and $(\R^d)^v$ to $\R$ that we specify individually later on.
Finally, for a function that is unbounded or not Lipschitz, we set its $\|\cdot\|_{\infty}$ norm or Lipschitz constant equal to infinity.

\begin{Definition}
	\label{gen}
	Let $\mathcal{I}$ be an index set as in Definition \ref{index}, $X=(X_t)_{t\in \mathcal{I}}$ be a process with values in $\R^d$ and $\Psi$ a function from $\overline{\R_+^6}$ to $\R_+$ non-decreasing in all arguments. The process $X$ is called $\Psi$-weakly dependent if there exists a sequence of coefficients $\epsilon=(\epsilon(r))_{r \in \R_+}$ converging to $0$ and satisfying the following inequality
	
	\begin{equation}
		\label{def}
		|Cov(F(X_{i_1},\ldots,X_{i_u}),G(X_{j_1},\ldots,X_{j_v}))|\leq c \, \Psi(\|F\|_{\infty},\|G\|_{\infty},Lip(F),Lip(G),u,v) \, \epsilon(r)
	\end{equation}
	for all
	\begin{equation*}
		\left\{
		\begin{array}{l}
			(u,v) \in \N^* \times \N^*;\\
			r\in \R_+; \\
			I_u=\{i_1,\ldots,i_u\} \subseteq{\mathcal{I}} \,\, \textrm{and}\,\, J_v=\{j_1,\ldots,j_v\} \subseteq \mathcal{I},  \textrm{such that}\,\,d(I_u,J_v) \geq r;\\
			\textrm{functions} \,\, F \colon (\R^{d})^u \to \R \,\, \textrm{and}\,\, G\colon (\R^{d})^v \to \R \,\,\textrm{belonging respectively to $\mathcal{F}$ and $\mathcal{G}$},
		\end{array}
		\right.
	\end{equation*}
	where $c$ is a constant independent of $r$.
\end{Definition}
W.l.o.g. we always consider $\epsilon$ as non-increasing sequence of coefficients. \\

In \cite{BS98}, the first covariance inequality for Lipschitz functions of positively or negatively associated random fields appears in the literature. Since this result, other covariance inequalities have been determined for functions $F$ and $G$ being either bounded Lipschitz or bounded measurable functions of processes and random fields. In the latter set-up, Definition \ref{gen} encompasses the so-called weak dependence conditions as described in Definition 5.12 of \cite{BS07} and Definition 2.2 of \cite{DD08} for $\mathcal{I}=\mathbb{Z}, \mathbb{Z}^m$. Therefore, several sequences of coefficients $\epsilon$ satisfying Definition \ref{gen} are already well-known.

\begin{itemize}
	\item Let $\mathcal{F}=\mathcal{G}$ and $\mathcal{F}_u$ be the class of bounded Lipschitz functions from $(\R^d)^u$ to $\R$  with respect to the distance $\delta$ on $(\R^d)^u$ defined by
	\begin{equation}
		\label{dist}
		\delta(x^*,y^*)= \sum_{i=1}^u \|x_i-y_i\|,
	\end{equation}
	where $x^*=(x_1,\ldots,x_u)$ and $y^*=(y_1,\ldots,y_u)$ and $x_i,y_i \in \R^d$ for all $i=1,\ldots,u$.
	Then, $ Lip(F)=\sup_{x\neq y} \frac{|F(x)-F(y)|}{\| x_1-y_1 \|+\|x_2-y_2\|+ \ldots+ \|x_d-y_d\|}$. For 
	\begin{equation*}
		\label{eta}
		\Psi(\|F\|_{\infty},\|G\|_{\infty},Lip(F),Lip(G),u,v)=u Lip(F) \|G\|_{\infty} +v Lip(G) \|F\|_{\infty},
	\end{equation*}
	$\epsilon$ corresponds to the \textbf{$\eta$-coefficients} as defined by \cite{DL99}. An extension of this definition for $\mathcal{I}=\mathbb{Z}^m$ is given by \cite{DL02}. If instead, \begin{align*}
		\label{lambda}
		\Psi(\|F\|_{\infty},\|G\|_{\infty},Lip(F),Lip(G),u,v)=&u Lip(F) \|G\|_{\infty} +v Lip(G) \|F\|_{\infty}\\ &+uv Lip(F)Lip(G) \nonumber,
	\end{align*}
	then $\epsilon$ corresponds to the \textbf{$\lambda$-coefficients} as defined by \cite{DW07} for $\mathcal{I}=\Z$ and in Remark 2.1 of \cite{DD08} for $\mathcal{I}=\Z^m$ . 
	Moreover, for 
	\begin{equation*}
		\label{kappa}
		\Psi(\|F\|_{\infty},\|G\|_{\infty},Lip(F),Lip(G),u,v)= uv Lip(F) Lip(G),
	\end{equation*}
	$\epsilon$ corresponds to the \textbf{$\kappa$-coefficients},
	and, for
	\begin{equation*}
		\label{zeta}
		\Psi(\|F\|_{\infty},\|G\|_{\infty},Lip(F),Lip(G),u,v)=min(u,v) Lip(F) Lip(G),
	\end{equation*}
	$\epsilon$ corresponds to the \textbf{$\zeta$-coefficients} as defined by \cite{DL99}. The definition of $\zeta$-weak dependence for $\mathcal{I}=\mathbb{Z}^m$ can be found in \cite{BS01}.
	\item Let $\mathcal{F}_u$ be the class of bounded measurable functions from $(\R^d)^u$ to $\R$ and $\mathcal{G}_v$ be the class of bounded Lipschitz functions from $(\R^d)^v$ to $\R$ with respect to the distance $\delta$ defined in (\ref{dist}). Then, for
	\begin{equation*}
		\label{theta}
		\Psi(\|F\|_{\infty},\|G\|_{\infty},Lip(F),Lip(G),u,v)= v \| F\|_{\infty} Lip(G),
	\end{equation*}
	$\epsilon$ corresponds to the \textbf{$\theta$-coefficients} as defined by \cite{DD08}. An extension of this definition for $\mathcal{I}=\mathbb{Z}^m$ appears in Remark 2.1 of \cite{DD08}. Moreover, an alternative definition for this notion of dependence is given by \cite{DD03} for $\mathcal{F}_u$ as above and $\mathcal{G}_1$ the class of Lipschitz function from $\R^d$ to $\R$, for $\mathcal{I}=\Z$.
\end{itemize}

The extension to index sets $\mathcal{I}=\R,\R^m$ of the weak dependence notions described above is straightforward.

\begin{Remark} 
	The weak dependence conditions can all be alternatively formulated by further assuming that $F \in \mathcal{F}$ and $G \in \mathcal{G}$ are bounded by one. For more details on this issue, see \cite{DL99} and \cite{DD03}. Therefore, an alternative definition of $\Psi$-weak dependence exists where the function $\Psi$ in Definition (\ref{def}) does not depend on $\|F\|_{\infty}$ and $\|G\|_{\infty}$. In this case, $\|F\|_{\infty}$ and $\|G\|_{\infty}$ are always bounded by one and therefore omitted in the notation. 
\end{Remark}

\noindent
We now show that Definition \ref{gen} also encompasses strong mixing.
We first define the strong mixing coefficient of \cite{R56}. 

We suppose that $\mathcal{A}_1$ and $\mathcal{A}_2$ are sub-$\sigma$-fields of $\mathcal{A}$ and define
\[
\alpha(\mathcal{A}_1, \mathcal{A}_2):= \sup_{\scriptsize \begin{array}{l} 
		A \in \mathcal{A}_1\\
		B \in \mathcal{A}_2
\end{array}} |P(A \cap B) - P(A)P(B)|.
\]
Let $\mathcal{I}$ a set as in Definition \ref{index}, then a process $X=(X_t)_{t \in \mathcal{I}}$ with values in $\R^d$ is said to be $\alpha_{u,v}$-mixing for $u,v\in\N\cup\{\infty\}$ if 
\begin{equation}
	\label{sup}
	\alpha_{u,v}(r):=\sup \{ \alpha(\mathcal{A}_{\Gamma_1}, \mathcal{B}_{\Gamma_2}): \Gamma_1, \Gamma_2\subseteq \mathcal{I}, |\Gamma_1| \leq u, |\Gamma_2| \leq v, d(\Gamma_1,\Gamma_2) \geq r \},
\end{equation}
converges to zero as $r \to \infty$, where  $\mathcal{A}_{\Gamma_1}=\sigma( X_i : i \in \Gamma_1)$ and $\mathcal{B}_{\Gamma_2}=\sigma( X_j: j \in \Gamma_2)$. If we let $\alpha(r)=\alpha_{\infty,\infty}(r)$, it is apparent that $\alpha_{u,v}(r) \leq \alpha(r)$. If $\alpha(r) \to 0$ as $r \to \infty$, then $X$ is simply said to be \textbf{$\alpha$-mixing}. For a comprehensive discussion on the coefficients $\alpha_{u,v}(r)$, $\alpha(r)$ and their relation to other strong mixing coefficients we refer to \cite{B07}, \cite{B88}, and \cite{D98}.

\begin{Proposition}
	\label{mixing2}
	Let $\mathcal{I}$ be a set as in Definition \ref{index} and $X=(X_t)_{t\in \mathcal{I}}$ be a process with values in $\R^d$ and $\mathcal{F}=\mathcal{G}$ where $\mathcal{F}_u$ is the class of bounded measurable functions from  $(\R^d)^u$ to $\R$. $X$ is $\alpha$-mixing if and only if there exists a sequence $(\epsilon(r))_{r \in \R_+}$ converging to $0$ such that
	
	\begin{equation}
		\label{mixalpha}
		|Cov(F(X_{i_1},\ldots,X_{i_u}),G(X_{j_1},\ldots,X_{j_v}))| \leq c \,  \Psi(\|F\|_{\infty},\|G\|_{\infty},Lip(F),Lip(G),u,v) \, \epsilon(r),
	\end{equation}
	where
	\begin{equation}
		\label{alpha}
		\Psi(\|F\|_{\infty},\|G\|_{\infty},Lip(F),Lip(G),u,v)=\|F\|_{\infty}\|G\|_{\infty},
	\end{equation}
	for all
	\begin{equation*}
		\left\{
		\begin{array}{l}
			(u,v) \in \N^* \times \N^*;\\
			r\in \R_+; \\
			I_u=\{i_1,\ldots,i_u\} \subseteq \mathcal{I} \,\, \textrm{and}\,\, J_v=\{j_1,\ldots,j_v\} \subseteq \mathcal{I},  \textrm{such that}\,\,d(I_u,J_v) \geq r\\
			\textrm{functions} \,\, F \colon (\R^{d})^u \to \R \,\, \textrm{and}\,\, G\colon (\R^{d})^v \to \R \,\,\textrm{belonging respectively to $\mathcal{F}$ and $\mathcal{G}$}
		\end{array}
		\right.
	\end{equation*}
	and where $c$ is a constant independent of $r$. 
	
\end{Proposition}

\begin{proof}
	Set $\mathcal{A}_{I_u}=\sigma(X_i: i \in I_u)$ and $\mathcal{B}_{J_v}=\sigma(X_j: j \in I_v)$.
	For arbitrary $(u,v) \in \N^* \times \N^*$ and $r \in \R_+$, let $I_u=\{i_1,\ldots,i_u\}$ and $J_v=\{j_1,\ldots,j_v\}$ be arbitrary subsets of $\mathcal{I}$ such that $d(I_u,J_v) \geq r$. Moreover, choose arbitrary $F \in \mathcal{F}_u$ and $G \in \mathcal{G}_v$, by Theorem 17.2.1 in \cite{IL}, it holds that
	\[
	|Cov(F(X_{i_1},\ldots,X_{i_u}),G(X_{j_1},\ldots,X_{j_v}))| \leq 4 \, \alpha(\mathcal{A}_{I_u},\mathcal{B}_{I_v}) \, \|F\|_{\infty}\|G\|_{\infty}.
	\]
	
	Definition (\ref{sup}) immediately implies that the right hand side of the inequality above is smaller than or equal to $4 \alpha(r) \,\|F\|_{\infty} \,\|G\|_{\infty}$. Hence, if $X$ is $\alpha$-mixing then (\ref{mixalpha}) holds with $\epsilon(r)=\alpha(r)$ and $c=4$.
	
	We assume now that the sequence $X$ is $\Psi$-weakly dependent with $\Psi$ given by (\ref{alpha}). By Theorem 17.2.1 in \cite{IL} and Remark 3.17(ii) in \cite{B07}, we can rewrite the definition of the $\alpha$-coefficients as
	\[
	\alpha(r)= \sup_{\scriptsize\begin{array}{l} 
			\Gamma_1, \Gamma_2 \subseteq \mathcal{I}\\ 
			|\Gamma_1|< \infty, |\Gamma_2|<\infty\\ 
			d(\Gamma_1,\Gamma_2) \geq r
	\end{array}}\alpha(\mathcal{A}_{\Gamma_1},\mathcal{A}_{\Gamma_2}) 
	\]
	\begin{equation}
		=\sup_{(u,v) \in \N\times \N} \sup_{\scriptsize \begin{array}{l}I_u, J_v \subseteq \mathcal{I}\\ d(I_u,J_v) \geq r \end{array} } \sup_{\scriptsize\begin{array}{l}F \in \mathcal{F}_u\\ G \in \mathcal{G}_v \end{array}} \Big \{ \frac{1}{4 \|F\|_{\infty} \|G\|_{\infty}} |Cov(F(X_{i_1},\ldots,X_{i_u}),G(X_{j_1},\ldots,X_{j_v}))| \Big \}.
	\end{equation}
	Hence,
	\[
	\alpha(r) \leq \frac{c}{4} \epsilon(r).
	\]
	If $X$ is $\Psi$-weakly dependent, then $X$ is $\alpha$-mixing.

\end{proof}

\begin{Remark}[$\theta$-lex weak dependence]
\label{lex}
The novel definition of $\theta$-lex weak dependence on $\mathcal{I}=\R^m$ appearing in \cite{CSS20} can be obtained by a slight modification of Definition 1.2. We use the notion of lexicographic order on $\R^m$: for distinct elements $y=(y_1,\ldots,y_m)\in\R^m$ and $z=(z_1,\ldots,z_m)\in\R^m$ we say $y<_{lex}z$ if and only if $y_1<z_1$ or $y_p<z_p$ for some $p\in\{2,\ldots,m\}$ and $y_q=z_q$ for $q=1,\ldots,p-1$. 
	\begin{itemize}
		\item Let $\mathcal{F}_u$ be the class of bounded measurable functions from $(\R^d)^u$ to $\R$ and $\mathcal{G}_1$ be the class of bounded Lipschitz functions from $\R^d$ to $\R$ with respect to the distance $\delta$ defined in (\ref{dist}). Moreover, let $I_u=\{i_1, \ldots, i_u\} \subset \R^m$, and $j \in \R^m$ be such that $i_s<_{lex} j$ for all $s=1,\ldots,u$, and $dist(I_u,j)\geq r$. Then, inequality (\ref{def}) holds for $\Psi(\|F\|_{\infty},\|G\|_{\infty},Lip(F),Lip(G),u,1)= \| F\|_{\infty} Lip(G),$ and $\epsilon$ corresponds to the \textbf{$\theta$-lex-coefficients}. 
	\end{itemize}
For $\mathcal{I}=\mathbb{Z}^m$, this notion of dependence is more general than $\alpha_{\infty,1}$-mixing as defined in (\ref{sup}), i.e., it applies to a broader class of models. Further, for $\mathcal{I}=\Z$,  $\theta$-lex weak dependence is more general than the notion of $\alpha$-mixing. We refer the reader to Section 2 of \cite{CSS20} for a complete introduction to $\theta$-lex weak dependence and its properties.
\end{Remark}

\begin{Remark}[Association]
	Association offers a complementary approach to the analysis of processes and random fields; see \cite{BS07} for a comprehensive introduction on this topic. Moreover, association is equivalent to $\zeta$-weak dependence under the assumptions of Lemma 4 in \cite{DL99}.
\end{Remark}

\section{Strong mixing and weak dependence coefficients under renewal sampling}
\label{sec2}
Let $X$ be a strictly stationary $\R^d$-valued process, i.e. for all $n \in \N$ and all $t_1,\ldots,t_n \in \mathcal{I}$ it holds that the finite dimensional distributions (indicated by $\mathcal{L}(\cdot)$) are shift invariant
\[
\mathcal{L}(X_{t_1+h},\ldots,X_{t_n+h})=\mathcal{L}(X_{t_1},\ldots,X_{t_n})\,\,\,\forall h \in \mathcal{I}.
\]
We want to investigate the asymptotic dependence of $X$ sampled at a renewal sequence. 

We use in the paper a definition of renewal process based on the sequence $\tau$, see \cite{H74}, and that agrees with the definition of a two-sided L\'evy process (cf. pg. 124, \cite{A}) in the Poisson case. Similar sampling schemes are used, e.g., by \cite{LM92}, \cite{CR07}, and \cite{ASM04}.

\begin{Definition}
	\label{inter_arrival}
	Let $\mathcal{I} \subseteq \R^m$ be a set as in Definition \ref{index} and $\tau=(\tau_i)_{i \in \Z\setminus \{0\}}$ be an $\mathcal{I}$-valued sequence of non-negative (component-wise) i.i.d. random vectors with distribution function $\mu$ such that $\mu\{0\} <1$.
	For $i \in \Z$, we define an $\mathcal{I}$-valued stochastic process $(T_i)_{i \in \Z}$ as 
	\begin{align}
		\label{eq1}
		T_0:=0 \quad \text{and} \quad T_i := \begin{cases}
			\,\,\,\,\,\sum_{j=1}^i \tau_j \, , \quad &\,\,\,\,\,i \in \N , \\
			-\sum_{j=i}^{-1} \tau_j \, , \quad &-i \in \N.
		\end{cases}
	\end{align}
	The sequence $(T_i)_{i \in \Z}$ is called a renewal sampling sequence. When $\mathcal{I} \subset \R$, we call $\tau$ the sequence of the inter-arrival times.
\end{Definition}

\begin{Definition}
	\label{ass2}
	Let $X=(X_t)_{t \in \mathcal{I}}$ be a process with values in $\R^d$ and let $(T_i)_{i \in \Z}$ be a renewal sampling sequence independent of $X$. We define the sequence $Y=(Y_i)_{i\in \Z}$ as the stochastic process with values in $\R^{d+1}$ given by 
	\begin{equation}
		\label{eq2}
		Y_i=\left (  X_{T_i}, T_i-T_{i-1} \right)^{\top}.
	\end{equation}
	We call $X$ the underlying process and $Y$ the renewal sampled process.
\end{Definition}

\begin{Remark}[Independence of $T$ and $X$]
	\label{time}
	The assumption of independence between the stochastic process $T$, modeling a random sampling scheme, and $X$ is reasonable when working with time series whose records are not event-triggered. For example, transaction level data, see \cite{H12} for a survey, are records of trade or transactions occurring when a buyer and a seller agree on a price for a security (triggering event). Even if these data should not be modeled by assuming that $T$ and $X$ are independent, this assumption is broadly used in the literature analyzing financial data; see, for instance, \cite{ASM04}, \cite{AM08}, and \cite{HY05}. Time series that are not event-triggered can, for example, be determined starting from the below data sets.
	
	\begin{itemize}
		\item Modern health monitoring systems like smartphones or wearable devices as smartwatches enable monitoring the health conditions of patients by measuring heart rate, electrocardiogram, body temperature, among other information, see \cite{Medical}. These measurements are records on a discrete-time grid, mostly irregularly distributed. In these cases, observation times depend on the measuring instrument (typically sensors), i.e., on a random source independent of the process $X$, as observed by \cite{BB10}. In this context, the hypothesis of independence of $T$ and $X$ is entirely realistic. 
		
		\item Measurements from continuous spatio-temporal random fields such as temperature, vegetation, or population are nowadays recorded over a set of moving or fixed locations in space and time, typically not regularly distributed. These data sets are called \textit{point reference or raster data} and are analyzed, for example, in earth science. Also, GPS data, e.g., of a taxi, which periodically transmit the location of an object over time, are an example of spatio-temporal data sets called \textit{trajectory data} which are typically irregularly distributed in space and time. We refer the reader to the survey of \cite{WCY19} and the references therein for an account of the data sets above and their practical relevance. The hypothesis of independence of $T$ and $X$ seems realistic for these data because their sampling in space-time depends on the instrument used to record them.
	\end{itemize}
	
\end{Remark}
\noindent
In the following theorem, we work with the class of functions defined in (\ref{set}) and 
\begin{equation}
	\label{set2}
	\mathcal{\tilde{F}}=\bigcup_{u \in \N} \mathcal{\tilde{F}}_u \,\,\, \quad \textrm{and} \quad \,\,\, \mathcal{\tilde{G}}=\bigcup_{v \in \N} \mathcal{\tilde{G}}_v,
\end{equation} 
where $\mathcal{\tilde{F}}_u$ and $\mathcal{\tilde{G}}_v$ are respectively two classes of measurable functions from $(\R^{d+1})^u$ to $\R$ and $(\R^{d+1})^v$ to $\R$ which can be either bounded or bounded Lipschitz.

\begin{Theorem}
	\label{main}
	Let $Y=(Y_i)_{i \in \Z}$ be a renewal sampled process with the underlying process $X$ being strictly stationary and $\Psi$-weakly dependent with coefficients $\epsilon$. Then, $Y$ is a strictly stationary process, and there exists a sequence $\mathcal{E}$ such that
	
	\begin{equation*}
		\label{bound}
		|Cov(\tilde{F}(Y_{i_1},\ldots,Y_{i_u}),\tilde{G}(Y_{j_1},\ldots,Y_{j_v}))| \leq C  \, \Psi(\|\tilde{F}\|_{\infty},\|\tilde{G}\|_{\infty},Lip(\tilde{F}),Lip(\tilde{G}),u,v)   \, \mathcal{E}(n)
	\end{equation*}
	
	for all
	\begin{equation*}
		\left\{
		\begin{array}{l}
			(u,v) \in \N^* \times \N^*;\\
			n\in \N; \\
			\{i_1,\ldots,i_u\} \subseteq \Z \,\, \textrm{and}\,\, \{j_1,\ldots,j_v\} \subseteq \Z, \\ \textrm{with}\,\, i_1\leq \ldots \leq i_u < i_u+n\leq j_1\leq \ldots\leq j_v; \\
			\textrm{functions} \,\, \tilde{F} \colon (\R^{d+1})^u \to \R \,\, \textrm{and}\,\, \tilde{G}\colon (\R^{d+1})^v \to \R \,\,\textrm{belonging to $\mathcal{\tilde{F}}$ and $\mathcal{\tilde{G}}$},
		\end{array}
		\right.
	\end{equation*}
	where $C$ is a constant independent of $n$.
	Moreover, 
	\begin{equation}
		\label{coef}
		\mathcal{E}(n)=\int_{\mathcal{I}} \epsilon(\|r\|) \, \mu^{*n}(dr),
	\end{equation}
	\noindent	
	where $\mu^{*0}$ is the Dirac delta measure in zero, and, $\mu^{*n}$ is  the n-fold convolution of $\mu$ for $n \geq 1$.
\end{Theorem}

\begin{proof}[Proof of Theorem \ref{main}]
	$Y$ is a strictly stationary process by Proposition 2.1 in \cite{BC18}.
	Consider arbitrary fixed $(u,v) \in \mathbb{N}^* \times \mathbb{N}^*$, $n \in \N$, $\{i_1,\ldots,i_u\} \subseteq \Z$ and $\{j_1,\ldots,j_v\} \subseteq \Z$ with $i_1\leq\ldots\leq i_u \leq i_u+n \leq j_1\leq \ldots\leq j_v$, and functions $\tilde{F} \in \tilde{\mathcal{F}}$ and $\tilde{G} \in \tilde{\mathcal{G}}$. W.l.o.g. let us consider throughout that $i_1 >0$. Then, by
	conditioning with respect to the sequence of the inter-arrival times $\tau$ and using the law of total covariance (cf. Proposition A.1 in \cite{C19}), we obtain that
	
	\begin{align}
		&|Cov(\tilde{F}(Y_{i_1},\ldots,Y_{i_u}),\tilde{G}(Y_{j_1},\ldots,Y_{j_v})) | \nonumber \\
		&\leq | \E(Cov(\tilde{F}(Y_{i_1},\ldots,Y_{i_u}), \tilde{G}(Y_{j_1},\ldots,Y_{j_v})| \tau_i : i=1,\ldots,j_v ))| \label{p1} \\
		&+ | Cov(\E(\tilde{F}(Y_{i_1},\ldots,Y_{i_u})| \tau_i: i=1,\ldots,j_v), \E(\tilde{G}(Y_{j_1},\ldots,Y_{j_v})|\tau_i : i=1,\ldots,j_v ) ) |. \label{p2}
	\end{align}
	
	Let us first discuss the summand (\ref{p2}). The term
	\[
	\E(\tilde{F}(Y_{i_1},\ldots,Y_{i_u})| \tau_i:i=1,\ldots,j_v)= \E(\tilde{F}(Y_{i_1},\ldots,Y_{i_u})| \tau_i:i=1,\ldots,i_u)
	\]
	because $\tilde{F}(Y_{i_1},\ldots,Y_{i_u})$	is independent of $\{\tau_i :i=i_{u}+1,\ldots,j_v\}$.
	On the other hand,
	\begin{align*}
		&\E(\tilde{G}(Y_{j_1},\ldots,Y_{j_v})|\tau_i : i=1,\ldots,j_v )\\
		&=\E(\tilde{G}((X_{T_{i_u}+\sum_{i=i_u+1}^{j_1} \tau_i}, \tau_{j_1})^{\prime},\ldots,(X_{T_{i_u}+\sum_{i=i_u+1}^{j_v} \tau_i}, \tau_{j_v})^{\prime})|\tau_i : i=1,\ldots,j_v ),
	\end{align*}
	and, by stationarity of the process $X$ and the i.i.d property of  $(\tau_i)_{i \in \Z \setminus \{0\}}$, it is equal to 
	\begin{align*}
		&\E(\tilde{G}((X_{\sum_{i=i_u+1}^{j_1} \tau_i}, \tau_{j_1})^{\prime},\ldots,(X_{\sum_{i=i_u+1}^{j_v} \tau_i}, \tau_{j_v})^{\prime})|\tau_i : i=1,\ldots,j_v ) \\
		&= \E(\tilde{G}((X_{\sum_{i=i_u+1}^{j_1} \tau_i}, \tau_{j_1})^{\prime},\ldots,(X_{\sum_{i=i_u+1}^{j_v} \tau_i}, \tau_{j_v})^{\prime})|\tau_i : i=i_u+1,\ldots,j_v)
	\end{align*}
	because of the independence between $\{(X_{\sum_{i=i_u+1}^{j_1} \tau_i}, \tau_{j_1})^{\prime},\ldots,(X_{\sum_{i=i_u+1}^{j_v} \tau_i}, \tau_{j_v})^{\prime}\}$ and \\
	$\{ \tau_i : i=1,\ldots,i_u \}$.
	Thus, the summand (\ref{p2}) is equal to zero because
	\[
	\E(\tilde{F}(Y_{i_1},\ldots,Y_{i_u})| \tau_i:i=1,\ldots,i_u),\]
	and
	\[
	\E(\tilde{G}((X_{\sum_{i=i_u+1}^{j_1} \tau_i}, \tau_{j_1})^{\prime},\ldots,(X_{\sum_{i=i_u+1}^{j_v} \tau_i}, \tau_{j_v})^{\prime})|\tau_i : i=i_u+1,\ldots,j_v) ) 
	\]
	are independent.
	
	The summand (\ref{p1}) is less than or equal to
	
	\begin{align*}
		\int_{\mathcal{I}^{j_v}} \Big | Cov(\tilde{F}((X_{\sum_{i=1}^{i_1} s_i}, s_{i_1})^{\prime},\ldots,&(X_{\sum_{i=1}^{i_u} s_i}, s_{i_u})^{\prime}), \tilde{G}((X_{\sum_{i=1}^{j_1} s_i}, s_{j_1})^{\prime},\ldots\\
		&\ldots,(X_{\sum_{i=1}^{j_v} s_i}, s_{j_v})^{\prime}) )  \Big | \,   d\mathbb{P}_{\{\tau_i:i=1,\ldots,j_v\}}(s_1,\ldots,s_{j_v}),
	\end{align*}
	where $\mathbb{P}_{\{\tau \}}$ indicates the joint distribution of the inter-arrival times sequence $\tau$. For a given $(s_{i_1},\ldots,s_{j_v}) \in \mathcal{I}^{j_v}$, we have that $\tilde{F}( (\cdot, s_{i_1}),\ldots,(\cdot,s_{i_u})) \in \mathcal{F}$ and $\tilde{G}((\cdot, s_{j_1}),\ldots,(\cdot, s_{j_v})) \in \mathcal{G}$. $X$ is a $\Psi$-weakly dependent process, then the above inequality is less than or equal to
	\[
	\int_{\mathcal{I}^{j_1-i_u}} C \, \Psi(\|\tilde{F}((\cdot, s_{i_1}),\ldots,(\cdot,s_{i_u}))\|_{\infty},\|\tilde{G}((\cdot, s_{j_1}),\ldots,(\cdot, s_{j_v}))\|_{\infty},Lip(\tilde{F}((\cdot, s_{i_1}),\ldots,(\cdot,s_{i_u}))),
	\]
	\[
	Lip(\tilde{G}((\cdot, s_{j_1}),\ldots,(\cdot, s_{j_v}))),u,v) \, \epsilon \Big(\Big \|\sum_{i=i_u+1}^{j_1} s_i \Big \|\Big)   d\mathbb{P}_{\{\tau_i:i=1,\ldots,j_v\}}(s_1,\ldots,s_{j_v}),
	\]
	and, because the sequence $\{\tau_i: i=i_u+1,\ldots,j_1\}$ is independent of the sequence $\{\tau_i: i=1,\ldots,i_u,j_1 +1,\ldots,j_v\}$ , the integral above is less than or equal to
	\[
	\int_{\mathcal{I}^{j_1-i_u}} C \, \Psi(\|\tilde{F}\|_{\infty},\|\tilde{G}\|_{\infty},Lip(\tilde{F}),Lip(\tilde{G}),u,v) \, \epsilon \Big(\Big\|\sum_{i=i_u+1}^{j_1} s_i \Big \|\Big) d\mathbb{P}_{\{\tau_i:i=i_u+1,\ldots,j_1\}}(s_{i_u+1},\ldots,s_{j_1}).
	\]
	We have that $j_1-i_u \geq n$ and that w.l.o.g. the coefficients $\epsilon$ are non increasing. Thus, we can conclude that the integral above is less than or equal to  
	\[
	C \, \Psi(\|\tilde{F}\|_{\infty},\|\tilde{G}\|_{\infty},Lip(\tilde{F}),Lip(\tilde{G}),u,v) \int_{\mathcal{I}} \epsilon(\|r\|) \mu^{*n}(dr).
	\]

\end{proof}

\begin{Corollary}
	\label{conclusion}
	Let the assumptions of Theorem \ref{main} hold. If the coefficients (\ref{coef}) are finite, and converge to zero as $n$ goes to infinity, then $Y$ is $\Psi$-weakly dependent with coefficients $\mathcal{E}$. 
\end{Corollary}
\begin{proof}[Proof of Corollary \ref{conclusion}]
It directly follows by Definition \ref{gen}.
\end{proof}

\begin{Remark}
	\label{charlot}
	\cite{CR07} obtain for a renewal process $T$ independent of $X$, $\alpha$-coefficients related to the process $(X_{T_i})_{i \in \Z}$ equal to $\E[\alpha(T_n)]$ which corresponds to (\ref{coef}) for all $n \in \N$. Their results also extend to renewal processes $T$ having inter-arrival time sequence $\tau$ which is itself $\alpha^{\prime}$-mixing with coefficients equal to $\E[\alpha(T_n)]+\alpha^{\prime}$. However, the techniques involved in their proof exploit the definition of $\alpha$-mixing coefficients as given in (\ref{sup}) and are not directly applicable in the case of weakly dependent processes. Moreover, the authors do not discuss how to obtain the inheritance of strong mixing, i.e., that the obtained $\alpha$-coefficients needs to be finite and converge to zero as $n$ goes to infinity. 
	
	We further explore this issue by discussing in Section \ref{expdecay} and \ref{powerdecay} several examples of sampling schemes for which the assumptions of Corollary \ref{conclusion} are satisfied.
\end{Remark}
\noindent
At last, concerning $\theta$-lex weak dependence defined in Remark \ref{lex}, we obtain the following result.

\begin{Corollary}
	\label{lex_cor}
	Let $X$ be a strictly stationary and $\theta$-lex weakly dependent random field defined on $\R^m$, and  $T$ be a renewal sampling sequence independent of $X$ with values in $\R^m$.
	Then $Y$ is a strictly stationary process, and there exists a sequence $\mathcal{E}$ such that
	\begin{equation}
		\label{theta_cov}
		|Cov(\tilde{F}(Y_{i_1},\ldots,Y_{i_u}),\tilde{G}(Y_{j}))| \leq C  \, \|F\|_{\infty} Lip(G) \, \mathcal{E}(n)
	\end{equation}
	
	for all
	\begin{equation*}
		\left\{
		\begin{array}{l}
			(u,v) \in \N^* \times \N^*;\\
			n\in \N; \\
			\{i_1,\ldots,i_u\} \subseteq \Z \,\, \textrm{and}\,\, j \in \Z, \\ \textrm{with}\,\, i_1\leq \ldots \leq i_u < i_u+n\leq j; \\
			\textrm{functions} \,\, \tilde{F} \colon (\R^{d+1})^u \to \R \,\, \textrm{and}\,\, \tilde{G}\colon \R^{d+1} \to \R \,\,\textrm{belonging to $\mathcal{\tilde{F}}$ and $\mathcal{\tilde{G}}$},
		\end{array}
		\right.
	\end{equation*}
	where $C$ is a constant independent of $n$, and $\mathcal{E}$ are defined in (\ref{coef}).
\end{Corollary}  

\begin{proof}[Proof of Corollary \ref{lex_cor}]
The sequence $\tau$ is a sequence of non-negative i.i.d random vectors, then $T_{i_1} \leq_{lex} \ldots \leq_{lex} T_{i_u} \leq_{lex} T_j$. Hence, stationarity of $Y$ follows from Proposition 2.1 in \cite{BC18}, and the covariance inequality (\ref{theta_cov}) holds by following the line of proof in Theorem \ref{main}.

\end{proof}
Note that, if the inequality (\ref{theta_cov}) holds and the coefficients (\ref{coef}) are finite and converge to zero as $n$ goes to infinity, then $Y$ is a $\theta$-weakly dependent process as defined in \cite{DD03}.

\section{Explicit bounds for $\Psi$-coefficients}
\label{sec3}

In this section, we consider renewal sampling of $X=(X_t)_{t \in \R}$. Therefore, the inter-arrival times are a sequence of non-negative i.i.d random variables with values in $\R$.  

We first show that if $X$ is $\Psi$-weakly dependent and admits exponential or power decaying coefficients $\epsilon$, then $Y$ is, in turn, $\Psi$-weakly dependent and its coefficients $\mathcal{E}$ preserve (at least asymptotically) the decay behavior of $\epsilon$.
This result directly enables the application of the limit theory for a vast class of $\Psi$-weakly dependent processes $Y$, of which we present several examples throughout the section.

In fact, central limit theorems for a $\Psi$-weakly dependent process $X$ typically hold under sufficient conditions of the following type:
$\E[\|X_0\|^{\delta}]< \infty$ for some $\delta >0$ and the coefficients $\epsilon$ satisfy a condition of the form
\begin{equation}
	\label{suff}
	\sum_{i=1}^\infty \epsilon(n)^{A(\delta)}  < \infty,
\end{equation}
where $A(\delta)$ is a certain function of $\delta$. 
If $X$ admits coefficients $\epsilon$ with exponential or sufficiently fast power decay, then conditions of type (\ref{suff}) are satisfied. If, in turn, $Y$ is $\Psi$-weakly dependent with coefficients having exponential or sufficiently fast power decay, then conditions of type (\ref{suff}) are also satisfied under renewal sampling.

\subsection{Exponential decay}
\label{expdecay}

In terms of the Laplace transform of the inter-arrival times, we can obtain a general bound for the coefficients $(\mathcal{E}(n))_{n \in \N}$.

\begin{Proposition}
	\label{Laplace}
	Let $X=(X_t)_{t \in \R}$, $Y=(Y_i)_{i \in \Z}$ and $(T_i)_{i \in \Z}$ be as in Theorem \ref{main}. Let us assume that $\epsilon(r)\leq C\mathrm{e}^{-\gamma r}$ for $\gamma>0$ and denote the Laplace transform of the distribution function $\mu$ by 
	\[
	L_{\mu}(t)=\int_{\R^{+}} \mathrm{e}^{-tr} \, \mu(dr), \,\, t \in \R_{+}.
	\]
	Then, the process $Y$ admits coefficients
	\[\mathcal{E}(n)\leq C\Big(\frac{1}{L_{\mu}(\gamma)}\Big)^{-n}\] 
	which converge to zero as $n$ goes to infinity.
\end{Proposition}

\begin{proof}[Proof of Proposition \ref{Laplace}]
	We notice that $L_{\mu}(t) <1$ for $t>0$ and that $L_{\mu^{*n}}(t)=(L_{\mu}(t))^n$, cf. \cite[Proposition 2.6]{S13}.
	
	Using the result obtained in Theorem \ref{main}, we have that
	\[
	\mathcal{E}(n)=\int_{\R_+} \epsilon(r) \, \mu^{*n}(dr) \leq C \int_{\R_+} \mathrm{e}^{-\gamma r} \, \mu^{*n}(dr) = C L_{\mu^{*n}}(\gamma)
	\]
	\[
	=C (L_{\mu}(\gamma))^n.
	\]
\end{proof}

As a direct consequence, if $X$ is $\Psi$-weakly dependent and admits exponentially decaying coefficients, the assumptions of Corollary \ref{conclusion} holds and $Y$ inherits the asymptotic dependence structure of $X$ under renewal sampling.

\begin{Example}
	\label{ciao}
	If we have a renewal sampling with $\Gamma(\alpha,\beta)$-distributed inter-arrival times for
	$\alpha, \beta >0$, then $\mu^{*n}$ is the distribution function of a $\Gamma(n\alpha, \beta)$ distributed random variable. By Proposition \ref{Laplace}, 
	\[
	\mathcal{E}(n)=\int_{\R_+} \epsilon(r) \, \mu^{*n}(dr) \leq C \int_{(0,+\infty)} \mathrm{e}^{-\gamma r} \frac{\beta^{n\alpha}}{\Gamma(n\alpha)} r^{n\alpha-1} \mathrm{e}^{-\beta r} \, dr = C \Big( \frac{\gamma+\beta}{\beta} \Big)^{-n\alpha}.
	\]
	
	A special case of the coefficients above is obtained for Poisson sampling, i.e., $\mu=Exp(\lambda)$ with $\lambda>0$. In this case, $\mu^{*n}$ is the distribution function of a $\Gamma(n,\lambda)$ distributed random variable, and
	\[
	\mathcal{E}(n)=\int_{\R_+} \epsilon(r) \, \mu^{*n}(dr) \leq C \int_{(0,+\infty)} \mathrm{e}^{-\gamma r} \frac{\lambda^n}{\Gamma(n)} r^{n-1} \mathrm{e}^{-\lambda r} \, dr = C \Big( \frac{\lambda+\gamma}{\lambda} \Big)^{-n}.
	\] 
\end{Example}

\begin{Remark}
	If the process $X$ is $\Psi$-weakly dependent with exponentially decaying coefficients, then the equidistant sampled process $(X_i)_{i \in \Z}$ has exponentially decaying coefficients $\epsilon(n)\leq C \mathrm{e}^{-\gamma n}$	for $\gamma >0$ and $n \in \N$. By using the results in Example \ref{ciao}, we can design renewal sampling schemes such that the process $Y$ has coefficients $\mathcal{E}$ with faster decay rate than the sequence of coefficients $\epsilon$.
	\begin{itemize}
		\item  For $\Gamma(\alpha,\beta)$-distributed inter-arrival times, we obtain that the process $Y$ has  faster-decaying coefficients than $(X_i)_{i \in \Z}$ if the parameters $\alpha, \beta >0$ are chosen such that
		\[
		\Big( \frac{\gamma+\beta}{\beta}  \Big)^{\alpha} \geq \mathrm{e}^{\gamma}.
		\]
		\item In the case of Poisson sampling, the process $Y$ admits faster-decaying coefficients than $(X_i)_{i \in \Z}$ if
		\[
		\lambda \geq \frac{\gamma}{\mathrm{e}^{\gamma}-1}.
		\]

		The fraction appearing at the right-hand side of the inequality is less than $1$ for all $\gamma >0$. Therefore, because the average length of two adjacent observations is ruled by $\E[\tau_1]=\frac{1}{\lambda}$, we can design an (on average) lower sampling frequency scheme such that the coefficients $\mathcal{E}$ decay faster than $\epsilon$, by choosing $\frac{\gamma}{\mathrm{e^{\gamma}}-1} \leq \lambda < 1$.
		
	\end{itemize}
\end{Remark}

\subsection{Power decay}
\label{powerdecay}
We now assume that the underlying process $X$ is $\Psi$-weakly dependent with coefficients $\epsilon(r)\leq C r^{-\gamma}$ for $\gamma>0$.

We start with some concrete examples of inter-arrival time distributions $\mu$ (and therefore of renewal sampling sequences $T$), preserving the power decay of the coefficients $\epsilon$.

\begin{Example}
	\label{gamma}
	Let us consider renewal sampling with $\Gamma(\alpha,\beta)$ distributed inter-arrival times for $\alpha, \beta >0$. Then, $\mu^{*n}$ is a $\Gamma(n\alpha, \beta)$ distribution. Thus, 
	\begin{equation}	
		\label{tricomi}
		\mathcal{E}(n)=\int_{\R_+} \epsilon(r) \, \mu^{*n}(dr) \leq C \int_{(0,+\infty)} r^{-\gamma} \frac{\beta^{n\alpha}}{\Gamma(n\alpha)} r^{n\alpha-1} \mathrm{e}^{-\beta r} \, dr = C \beta^{\gamma} \frac{\Gamma(n\alpha-\gamma)}{\Gamma(n\alpha)}.
	\end{equation}
	For $n \to \infty$, and applying Stirling's series, see \cite{TE51}, we obtain that (\ref{tricomi}) is equal to $C \beta^{\gamma} n^{-\gamma} + O(n^{-\gamma-1}).$
	
	In the particular case of Poisson sampling, $\mu^{*n}$ is a $\Gamma(n,\lambda)$ distribution and 
	\[
	\mathcal{E}(n)=\int_{\R_+} \epsilon(r) \, \mu^{*n}(dr) \leq C \int_{(0,+\infty)} r^{-\gamma } \frac{\lambda^n}{\Gamma(n)} r^{n-1} \mathrm{e}^{-\lambda r} \, dr = C \lambda^{\gamma} \frac{\Gamma(n-\gamma)}{\Gamma(n)}\]
	\[
	=C \lambda^{\gamma} n^{-\gamma}(1+O(n^{-1})) = C \lambda^{\gamma} n^{-\gamma} + O(n^{-\gamma-1}),
	\]
	where the last equality holds as $n \to \infty$. 
\end{Example}

\begin{Example}
	\label{levy}
	We denote by $Levy(0,c)$ a L\'evy distribution, cf. pg. 28 \cite{Z86}, with location parameter $0$ and scale parameter $c$ (a completely skewed $\frac{1}{2}$-stable distribution). This distribution has infinite mean and variance. For $Levy(0,c)$ distributed inter-arrival times, we have that $\mu^{*n}$ is $Levy(0,cn)$. Thus,
	\[
	\mathcal{E}(n)=\int_{\R_+} \epsilon(r) \, \mu^{*n}(dr)\leq C\int_{\R_+} r^{-\gamma} \, \frac{(\frac{cn}{2})^{\frac{1}{2}}}{\Gamma(\frac{1}{2})} r^{-\frac{3}{2}} \mathrm{e}^{-\frac{cn}{2r}} \, dr = C  \frac{\Gamma(\frac{1}{2}+\gamma)}{(\frac{cn}{2})^{\gamma}\Gamma(\frac{1}{2})}= C \frac{\Gamma(\frac{1}{2}+\gamma)}{\Gamma(\frac{1}{2})} \Big(\frac{c}{2}\Big)^{-\gamma} n^{-\gamma}.
	\]
\end{Example}

\begin{Example}
	\label{igauss}
	We consider now the case where $\mu$ is an inverse Gaussian distribution with mean $m$ and shape parameter $\lambda$ (short  $IG(m,\lambda)$). We have that $\mu^{*n}$ is a $IG(nm,n^2\lambda)$ distribution and
	\begin{align}
		\mathcal{E}(n)= \int_{\R_+} \epsilon(r) \, \mu^{*n}(dr) &\leq C\int_{(0,+\infty)} r^{-\gamma} \Big( \frac{n^2 \lambda}{2\pi r^3} \Big)^{\frac{1}{2}} 
		\mathrm{e}^{-\frac{n^2\lambda(r-nm)^2}{2n^2m^2r}} \, dr \nonumber \\
		&= n C \Big( \frac{\lambda}{2\pi}  \Big)^{\frac{1}{2}} \mathrm{e}^{\frac{\lambda n}{m}} \int_{(0,+\infty)} r^{-\gamma-\frac{3}{2}} 
		\mathrm{e}^{-\frac{\lambda n}{2 m}\Big( \frac{r}{nm}+\frac{nm}{r}\Big)} \, dr \nonumber \\
		&= C \Big( \frac{\lambda}{2\pi}  \Big)^{\frac{1}{2}} m^{-\gamma-\frac{1}{2}} \, n^{-\gamma+\frac{1}{2}} \,\mathrm{e}^{\frac{\lambda n}{m}} \, 2 \, \mathcal{K}_{-\gamma-\frac{1}{2}}\Big(\frac{\lambda n}{m}\Big) \label{bessel} 
	\end{align}
	after applying the substitution $x :=\frac{r}{nm}$ and where $\mathcal{K}_{-\gamma-\frac{1}{2}}$ denotes a modified Bessel function of the third kind with order $-\gamma-\frac{1}{2}$.
	Using the asymptotic expansion for modified Bessel functions at pg. 171 in \cite{J82}, we obtain that  $\mathcal{K}_v(x)= (\frac{\pi}{2})^{\frac{1}{2}} x^{-\frac{1}{2}} \mathrm{e}^{-x}(1+O(x^{-1}))$. Thus, for $n \to \infty$, (\ref{bessel}) is equal to $\frac{C}{2}  m^{-\gamma}\, n^{-\gamma}+ O(n^{-\gamma-1}).$
	
\end{Example}

\begin{Example}
	\label{bern}
	Let the inter-arrival times follow a Bernoulli distribution with parameter $0 \leq p \leq 1$. Then, $\mu^{*n}$ is a $Bin(n, p)$ distribution. If $X$ admits coefficients $\epsilon(r)=C(1 \wedge r^{-\gamma})$ for $\gamma>0$, we have that $\mathcal{E}(n)$  	
	\begin{equation}
		\mathcal{E}(n)= \int_{\R_+} \epsilon(r) \, \mu^{*n}(dr) = C \Big( (1-p)^n + \sum_{j=1}^n j^{-\gamma} {n\choose j} p^j (1-p)^{n-j} \Big) \label{binom}.
	\end{equation}
	For $n \to \infty$, applying the asymptotic expansion proved in Theorem 1 by \cite{WW08}, we have that (\ref{binom}) is equal to $C (np)^{-\gamma} + O(n^{-\gamma-1})$.
\end{Example}

\begin{Example}
	\label{mass}
	Let us consider inter-arrival times such that $\mu([0,k))=0$ for a fixed $k >0$. Then, straightforwardly
	\[
	\mathcal{E}(n)=\int_{\R_+} \epsilon(r) \, \mu^{*n}(dr) \leq C (nk)^{-\gamma}.
	\] 
	
\end{Example}

\noindent
In Examples \ref{gamma}, \ref{igauss}, and \ref{bern} we obtain asymptotic bounds for the coefficients $\mathcal{E}$ whereas we have exact ones in Examples \ref{levy} and \ref{mass}. For a general inter-arrival time distribution we can just show that the coefficients $\mathcal{E}$ decay at least (asymptotically) with the same power. This result relies on the following Lemma.

\begin{Lemma}
	\label{lem}
	Let $\mu, \nu$ be two probability measures on $\R^{+}$ such that $\mu([0,b)) \leq \nu([0,b))$ for all $b >0$ and $f: \R^{+} \to \R^{+}$ be non-increasing. Then
	\[
	\int_{\R_+} f(r) \mu^{*n}(dr) \leq \int_{\R_+} f(r) \, \nu^{*n}(dr).
	\]
\end{Lemma}

\begin{proof}[Proof of Lemma  \ref{lem}]
	The proof follows by applying measure-theoretic induction.
\end{proof}

\begin{Proposition}
	\label{Bernoulli}
	Let $X=(X_t)_{t \in \R}$, $Y=(Y_i)_{i \in \Z}$ and $(T_i)_{i \in \Z}$ be as in Theorem \ref{main}. Let us assume that $\epsilon(r)\leq C r^{-\gamma}$ for $\gamma>0$.
	Let $a >0$ be a point in the support of $\mu$ such that $\mu([0,a))>0$, and set $p=\mu([a,\infty])$. Then, the process $Y$ admits coefficients $\mathcal{E}(n)\leq C (n a p)^{-\gamma}$ as $n \to \infty$.
\end{Proposition}

\begin{proof}[Proof of Proposition \ref{Bernoulli}]
	Let us assume w.l.o.g. that $\mu\neq \delta_a$ (otherwise Example \ref{mass} applies for any $a \in \R_{+}$), where $\delta_a$ denotes the Dirac-delta measure for $a \in \R_+$. Set $\nu= p \delta_{a}+(1-p) \delta_0$. The latter is a Bernoulli distribution that assigns probability $p$ to the inter-arrival time $a$ and $(1-p)$ to the one $0$. It follows that $\mu([0,b))\leq \nu([0,b))$ for all $b >0$. Then, by using Lemma \ref{lem}, the result in Example \ref{bern} and Theorem 1 in \cite{WW08}
	\begin{align*}
		\mathcal{E}(n) \leq C \int_{\R_+} r^{-\gamma} \mu^{*n}(dr) \leq C \int_{\R_+} r^{-\gamma} \nu^{*n}(dr) &=  C \Big( (1-p)^n + \sum_{j=1}^n (a j)^{-\gamma} {n\choose j} p^j (1-p)^{n-j} \Big) \\
		&= C (n a p)^{-\gamma} + O(n^{-\gamma- 1}),
	\end{align*}
	where the last inequality holds for $n \to \infty$.
\end{proof}

\begin{Remark}
Proposition \ref{Bernoulli} gives us an upper bound for the coefficients $\mathcal{E}$. This means that the 
true decay of the coefficients $\mathcal{E}$ could be faster, in general, than $n^{-\gamma}$. 
However, we have not found examples of sequences $\tau$ where this happens.  
Even for extremely heavily tailed inter-arrival time distributions like in Example \ref{levy}, we can just find an estimate from above of the coefficients of the renewal sampled process $Y$, i.e., $\mathcal{E}(n) \leq C n^{-\gamma}$ for large $n$, that has the same power decay as the coefficients $\epsilon$.
\end{Remark}

\noindent
Proposition \ref{Bernoulli} summarizes the results given in this section.
In fact, as long as $X$ is $\Psi$-weakly dependent such that there exists a $\gamma >0$ with $\epsilon(r) \leq C r^{-\gamma}$ then the assumptions of Corollary \ref{conclusion} are satisfied and $Y$ inherits the asymptotic dependence structure of $X$.
Note that Proposition \ref{Bernoulli} assures that $Y$ is $\Psi$-weakly dependent also when, for example, $\epsilon(r)=C\frac{1}{r log(r)}$ and then $\epsilon(r) \leq C n^{-1}$. 
Therefore, caution has to be exercised when checking conditions of type (\ref{suff}) for the process $Y$. 

\begin{Example}
	\label{special}
	Let us consider the sufficient condition for the applicability of the central limit theorem for $\kappa$-weakly dependent processes, see \cite{DW07}, where (\ref{suff}) holds with $A(\delta)=1$.
	If $X$ is a $\Psi$-weakly dependent process with coefficients $\epsilon(r)=C\frac{1}{r log(r)}$, then $Y$ is a $\Psi$-weakly dependent process with coefficients $\mathcal{E}(n) \leq \tilde{C} n^{-1}$ as $n \to \infty$ by applying Proposition \ref{Bernoulli}.
	We have that the coefficients $\epsilon(r)$ are summable and satisfy (\ref{suff}), but we do not know the summability of the coefficients $\mathcal{E}(n)$ as Proposition \ref{Bernoulli} just gives an upper bound of their value which is not summable.
\end{Example}

\section{Conclusion}
\label{sec4}

We assume that our sampling scheme is described by a renewal sequence $T$ independent of the process $X$ being weakly dependent or $\alpha$-mixing. We determine under which assumptions the process $Y=(X_{T_i},T_i-T_{i-1})$ is itself weakly dependent or $\alpha$-mixing. If $X$ admits exponential or power decaying coefficients, $Y$ inherits strong mixing or weak dependence, and its related coefficients preserve the exponential or power decay (at least asymptotically). Our general results enable the application of central limit theorems heavily used under equidistant sampling schemes to renewal sampled data.

Other sampling schemes are of great interest in practical applications and constitute a natural continuation of our work. For instance, sampling schemes where $T$ is a point process dependent on $X$, as observed in transaction-level financial data. Moreover, when analyzing data from continuous spatio-temporal random fields, the theory we have developed so far allows analyzing sampling along a self-avoiding walk that moves in non-negative coordinate directions. Another possible extension of our theory aims to study the random field sampling along a walk that moves in lexicographically increasing coordinate directions. 

\section*{Acknowledgment}
We would like to express our gratitude to the two anonymous reviewers and the editors for their many insightful comments and suggestions.

\bibliographystyle{apa}

\bibliography{preservation_Brandes_Curato_Stelzer_revision}

\end{document}